\documentclass[11pt]{article}

\usepackage{amsmath, amsfonts, amssymb}
\usepackage{mathrsfs}
\usepackage{theorem}
\usepackage{bm}

\newtheorem{thm}{Theorem}[section]

\newtheorem{lem}[thm]{Lemma}

{\theorembodyfont{\rmfamily} \newtheorem{rem}[thm]{Remark}}
{\theorembodyfont{\rmfamily} }

\newcommand{\ra}{\rightarrow}
\newcommand{\dis}{\displaystyle}

\def\R{\mathbb R}

\def\d{\text{\rm{d}}}
\def\E{\mathbb E}
\def\p{\mathbb P}\def\e{\text{\rm{e}}}

\def\la{\langle}
\def\raa{\rangle}
\def\La{\Lambda}
\def\veps{\varepsilon}

\def\de{\delta}

\def\rsdp{\textbf{RSDP}}

\def\S{\mathcal S}

\def\vsig{\varsigma}

\newcommand{\fin}{\hspace*{\fill}\rule{0.3em}{1ex}}
\newenvironment{proof}{{\bf \noindent Proof.}}{\fin}

\numberwithin{equation}{section}

\allowdisplaybreaks

\begin{document}

\title{Invariant measures and Euler-Maruyama's approximations of state-dependent regime-switching diffusions\footnote{Supported in
 part by    NNSFs of China No. 11431014, 11771327}}

\author{Jinghai Shao\footnote{Email: shaojh@bnu.edu.cn}\\[0.2cm]
Center for Applied Mathematics, Tianjin
University, Tianjin 300072, China}
\maketitle

\begin{abstract}
  Regime-switching processes contain two components: continuous component and discrete component, which can be used to describe a continuous dynamical system in a random environment. Such processes have many different properties than general diffusion processes, and much more difficulties are needed to be overcome due to the intensive interaction between continuous and discrete component. We give conditions for the existence and uniqueness of invariant measures for state-dependent regime-switching diffusion processes by constructing a new Markov chain to control the evolution of the state-dependent switching process. We also establish the strong convergence in the $L^1$-norm of the Euler-Maruyama's approximation and estimate the order of error. A refined application of Skorokhod's representation of jumping processes plays a substantial role in this work.
\end{abstract}



\section{Introduction}

The regime-switching diffusion processes have drawn much attention owing to the demand of modeling, analysis and computation of complex dynamical systems. Classical models using deterministic differential equations and stochastic differential equations alone are often inadequate, and many models having considered the random switching of the environment are extensively proposed and investigated in control engineering, queueing networks, filtering of dynamic systems, ecological and biological systems, mathematical finance etc. recently. This kind of process has been studied by Skorokhod \cite{Sk}, where it was called a process with a discrete component to emphasize the difference caused by the application of discrete topology for some component of the investigated process. Precisely, the regime-switching diffusion process (\textbf{RSDP}) concerned in this work have two components $(X(t),\La(t))$. $(X(t))$ is used to describe the continuous dynamical system satisfying the following stochastic differential equation (SDE):
\begin{equation}\label{1.1}
\d X(t)=b(X(t), \La(t))\d t+\sigma(X(t),\La(t))\d W(t),\quad t>0, X_0=x\in \R^n,\ \La(0)=i\in\S,
\end{equation}
where $b:\R^n\times\S\ra \R^n$, $\sigma:\R^n\times\S\ra \R^n\otimes\R^n$, and $\S=\{1,2,\ldots,N\}$ with $N<\infty$.  $(\La(t))$ is used to describe the switching of regimes or the change of environment in which $(X(t))$ lives. $(\La(t))$ is a jumping process on $\S$ with the transition rate satisfying
\begin{equation} \label{1.2}
  \p(\La(t+\Delta)=j|\La(t)=i,\ X(t)=x)=\begin{cases}
    q_{ij}(x)\Delta+o(\Delta), &j\neq i,\\
    1+q_{ii}(x)\Delta+o(\Delta),&j=i
  \end{cases}
\end{equation} provided $\Delta\downarrow 0$. When $q_{ij}(x)$ is independent of $x$ for all $i,\,j\in \S$, $(X(t),\La(t))$ is called a state-independent \textbf{RSDP} or a \textbf{RSDP} with Markovian switching. Otherwise, it is called a state-dependent \textbf{RSDP}.

Although the \textbf{RSDP}s are seemingly similar to the well-known diffusion processes with time-dependent coefficients, their properties are quite different from those of the usual diffusion processes. Compared with the diffusion process in a fixed environment, the \textbf{RSDP} owns much more complicated behavior. The random switching of the environment has essential impact on the properties of this system, for example, the properties of recurrence, stability, and tail behavior of the stationary distribution. Pinsky and Scheutzow in \cite{PS} constructed two examples on the half line, which showed that even if the \textbf{RSDP} in every fixed environment is recurrent (or transient), this process itself could be transient (or recurrent respectively) under certain random switching rate of the environment.  Similar phenomenon appears in the study of stability of the \textbf{RSDP}, and we refer to the works \cite{BBG1996, BBG1999,FC, KZY2007,SX}  and references therein for the study of stability of the \rsdp. The monographs \cite{MY} and \cite{YZ} provide good summaries of the recent progress in the study of  state-independent and state-dependent \rsdp s\, respectively.
As shown in \cite{SY}, \cite{Bar} for the Ornstein-Uhlenbeck process with Markovian switching, and in \cite{HS17} for the Cox-Ingeroll-Ross process with Markovian switching, the stationary distributions of the corresponding processes with switching could be heavy-tailed, but the stationary distributions of the processes without switching must be light-tailed. Therefore, the heavy-tailed empirical evidence promotes the application of models with regime-switching.

The recurrent property of \rsdp\, has been extensively investigated; see, for example, \cite{BL,CH,PP,SX,Sh15a,Sh15b} for the setting of state-independent switching processes, \cite{BL,CH,Sh15a} for the setup of bounded
state-dependent switching processes, \cite{MT} for the framework of
unbounded and state-dependent switching processes. So far, there are several
approaches to explore ergodicity for \rsdp s; see, for
instance, \cite{BL,Sh15a} via probabilistic coupling argument,
\cite{CH,MT} by weak Harris' theorem, \cite{PP,Sh15a,Sh15b} based on
the theory of M-matrix, Perron-Frobenius theorem and the Fredholm alternative. In particular, to study the
ergodicity and stability of \rsdp\, with infinitely countable regimes, we have proposed two methods in
\cite{SX,Sh15a,Sh15b}, i.e. finite partition method based on the M-matrix theory and the principal eigenvalue of bilinear forms method.

Recently, previously introduced \rsdp s have been extended in two directions: one is to extend SDEs driven by Brownian motion to those driven by general L\'evy processes (e.g. \cite{Xi, YX2010, XZ17}); anther is to extend SDEs to functional SDEs (e.g. \cite{M13, Sh17a, BSY17}) or the discrete switching process depending on the past of the continuous process in order to deal with the past dependence of the system in practice (e.g. \cite{NY}).

The purpose of current work is to study the existence of invariant measures and Euler-Maruyama's approximation of state-dependent \rsdp. For  \rsdp s with Markovian switching, these two problems have relatively been well studied. See, for instance, \cite{CH,Sh15a,BSY} for existence of invariant measures, \cite{YM04,MYY} for the numerical approximation of state-independent \rsdp\,under Lipschitz and non-Lipschitz conditions. However, these two problems for the state-dependent \rsdp s are not well studied. In \cite{YZ}, some types of Foster-Lyapunov conditions were given on the recurrence of state-dependent \rsdp s by viewing $(X(t),\La(t))$ as a special kind of jump-diffusions.
But, it is very hard to find suitable Lyapunov functions for state-dependent \rsdp s. In \cite{Sh15a}, we simplified the transition rate matrices of $(\La(t))$ by introducing a new transition rate matrix and its associated Markov chain, then used the M-matrix theory to give out a criterion on the recurrence of $(X(t),\La(t))$. The regime-switching systems are rather complicated, and it is usually impossible to get explicit solutions of such systems. So the numerical approximation is an important alternative of such systems. However, there was few work besides \cite{YMYC} on the numerical approximation of state-dependent \rsdp s due to the close interaction between the continuous component and the discrete component. In \cite{YMYC}, the weak convergence of numerical approximation was established by constructing a sequence of discrete-time Markov chains. This method is different to the usual time-discretizing Euler-Maruyama's approximation, and is difficult to obtain the order of error.
The main difficulty is that the evolution of $(\La(t))$ is much more complicated due to its dependence on the continuous-state process $(X(t))$, which makes the transition rate matrices of $(\La(t))$ are different for every step of jumps. Much care and more techniques need to be exercised to handle the mixture of $(X(t))$ and $(\La(t))$. In this work, we aim to establish the strong convergence of the time-discretizing Euler-Maruyama's approximation and estimate its order of error.

In this work, the existence and uniqueness of invariant measure for $(X(t),\La(t))$ is established by the convergence of the distribution of $(X(t),\La(t))$ in the Wasserstein distance as in \cite{CH} and \cite{Sh15b}. We construct the coupling by reflection of $(X(t),\La(t))$, and provide explicit conditions to guarantee this coupling to be successful. This result also weakens the conditions imposed in \cite{XS} on the successful coupling of state-dependent \rsdp s. Here, we show that the coupling is successful if the corresponding coupling process in at least one fixed environment is successful uniformly relative to the initial points. In \cite{XS}, it needs that  the corresponding coupling processes in every fixed environment is successful uniformly with respect to the initial points. An important technique in this procedure is the construction of an auxiliary Markov chain to control the evolution of the state-dependent jumping process $(\La(t))$ based on its Skorokhod's representation (see Lemma \ref{t-1.1}  and Lemma \ref{t-2.1} below).


Let $(X(t),\La(t))$ be the solution of \eqref{1.1} and \eqref{1.2} with additive noise, i.e. $\sigma(x,i)\equiv \sigma\in \R^{n\times n}$. In present work, we consider the following Euler-Maruyama's approximation of $(X(t),\La(t))$: for $\delta\in (0,1)$, define
\begin{align*}
  \d X^{\delta}(t)&=b(X^\delta(t),\La^\delta(t))\d t+\sigma \d W(t),\\
  \d \La^\delta(t)&=\int_{[0,M]}h(X^\delta(t),\La^\delta(t-),z)N_1(\d t,\d z),
\end{align*} with $(X^\delta(0),\La^\delta(0))=(X(0),\La(0))$, where $t_\delta=[t/\delta]\delta$, and $[t/\delta]$ denotes the integer part of $t/\delta$. Under some hypotheses, we show in Theorem \ref{t-3.3} that there exists some constant $C>0$ such that for $T>0$,
\[\E\big[\sup_{0\leq t\leq T}|X(t)-X^\delta(t)|\big]\leq C\delta^{\frac 12}.\]
To show this strong convergence, the main difficulty comes from the estimation of
\begin{equation}\label{s-t}
\int_0^t\p(\La(s)\neq \La^\delta(s))\d s, \quad t>0.
\end{equation}
Using Skorokhod's representations of $(\La(t))$ and $(\La^\delta(t))$, we show that the Lipschitz continuity of the transition rate function $x\mapsto q_{ij}(x)$ can yield that there is a constant $C>0$ such that
\begin{equation}\label{s-t-1}
\int_0^t\p(\La(s)\neq \La^\delta(s))\d s\leq C\delta^{\frac 12}+C\int_0^t\E|X(s)-X^\delta(s)|\d s.
\end{equation}
Due to the importance of the quantity \eqref{s-t} in the analysis of state-dependent regime-switching processes, this type of estimate \eqref{s-t-1} is of great interest by itself.

This paper is organized as follow. In Section 2, we investigate the existence of the invariant measure for state-dependent \rsdp s. We apply the coupling method to prove the convergence of the distributions of $(X(t),\La(t))$ in the Wasserstein distance to its unique invariant measure. We construct the coupling by reflection for \rsdp. To guarantee this coupling to be successful, we improve the result in \cite{XS} by providing weaker and more explicit conditions. Owing to the state-dependence, the transition rate matrices of the jumping process $(\La(t))$ may be different for every step of jumps. The usual technique to handle Markovian switching diffusions, i.e. ensuring first the discrete component meet together, then the continuous component meet together, does not work any more. For the state-dependent case, we have to make two components meet together at the same time. In order to control the state-dependent jumping process $(\La(t))$, we construct a state-independent Markov chain $(\bar \La(t))$ so that almost surely $\La(t)\leq \bar \La(t)$ for all $t\geq 0$ and provide explicit condition in terms of $(\bar \La(t))$ to control the the exponential functional of $(\La(t))$, i.e.
\[\E \e^{\int_0^t \lambda_{\La(s)}\d s}\]
where $\lambda:\S\ra \R$. The limitation of our construction is that the jumping process for each continuous-state $x$ should be of birth-death form, i.e. $q_{ij}(x)=0$ for any $i,\,j\in\S$, $|i-j|\geq 2$, and   $x\in \R^n$.

In Section 3, we explore the Euler-Maruyama's approximation for state-dependent \rsdp s. The key point is the estimate given in Lemma \ref{t-3.2}. The strong convergence of Euler-Maruyama's approximation is presented in Theorem \ref{t-3.3} with the order of error being $\sqrt{\delta}$. Note that this order of error consists with the order of error provided by \cite{YM04} for numerical approximation of Markovian regime-switching diffusion processes.

\section{Invariant measures}
Consider the state-dependent \rsdp\, $(X(t),\La(t))$ defined by \eqref{1.1} and \eqref{1.2}. The assumptions used in this work on the coefficients and transition rate matrix are collected as follows.

For the transition rate matrix $Q(x):=(q_{ij}(x))_{i,j\in\S}$, we shall use the following conditions:
\begin{itemize}
  \item[(Q1)] For each $x\in \R^n$, $(q_{ij}(x))$ is conservative and irreducible.
  \item[(Q2)] $H:=\max_{i\in\S}\sup_{x\in\R^n} q_i(x)<\infty$, where $q_{i}(x)=\sum_{j\neq i}q_{ij}(x)$ for $i\in\S$, $x\in\R^n$.
  \item[(Q3)] There exists a constant $c_q$ so that $|q_{ij}(x)-q_{ij}(y)|\leq c_q|x-y|$, $\forall$ $x,\,y\in\R^n$, \ $i,\,j\in \S$.
\end{itemize}

Concerning the coefficients of SDE \eqref{1.1}, we shall use the following conditions:
\begin{itemize}
  \item[\textbf{(A1)}] There exist constants $\alpha_i\in\R$, $i\in \S$,  such that
  \[2\la x-y,b(x,i)-b(y,i)\raa+2\|\sigma(x,i)-\sigma(y,i)\|_{\mathrm{HS}}^2\leq \alpha_i|x-y|^2, \ \ x,\,y\in\R^n,\ i\in \S.\]
  \item[\textbf{(A2)}] There exists a constant $C_1$ such that
  \[ |b(x,i)|+\|\sigma(x,i)\|_{\mathrm{HS}}\leq C_1,\ \ x\in \R^n,\ i\in \S.\]
  \item[\textbf{(A3)}] There exist constants $C_2>0$ such that
  \begin{equation*}\label{con-sig}
u^\ast \sigma(x,i)^\ast u\geq C_2, \quad \forall \ u\in \R^n, |u|=1, \ x\in \R^n, \,i\in\S.
\end{equation*}
\item[\textbf{(A4)}] There exist some state $i_0\in\S$, constants $p>2$, $C_3>0$ and $\beta\in \R$ such that
\begin{equation*}\label{con-b}
\la x-y, b(x,i_0)-b(y,i_0)\raa+\|\sigma(x,i_0)-\sigma(y,i_0)\|_{\mathrm{HS}}^2 \leq \beta |x-y|^2-C_3|x-y|^p,\ x,y\in\R^n.
\end{equation*}
\end{itemize}
The conditions (Q1)-(Q3) and \textbf{(A1)}-\textbf{(A2)} are used to guarantee the existence of unique non-explosive strong solution of \eqref{1.1} and \eqref{1.2} (cf. for example, \cite{Sh15c}). Besides, condition (Q3) also plays important role in the estimation of $\p\Big(\int_0^t\mathbf 1_{\{\La(s)\neq \La'(s)\}}\d s\Big)$ when studying numerical approximation of state-dependent \rsdp s. Condition \textbf{(A4)} is used to guarantee the constructed coupling processes of the state-dependent \rsdp\, to be successful, which improves the result in \cite{XS} on successful coupling in two aspects: first, the condition \textbf{(A4)} is more explicit than the condition (T1) in \cite{XS}, and hence is easier to be verified; second, in this work it only needs that \textbf{(A4)} holds for at least one state of $\S$,  however, the condition (T1) in \cite{XS} must hold for all states in $\S$.

Next, we introduce Skorokhod's representation of $\La(t)$ in terms of the Poisson random measure as in \cite[Chapter II-2.1]{Sk} or \cite{YZ}. For each $x\in \R^n$, we construct a family of intervals $\{\Gamma_{ij}(x); \ i,j\in \S\}$ on the half line in the following manner:
\begin{align*}
  \Gamma_{12}(x)&=[0,q_{12}(x))\\
  \Gamma_{13}(x)&=[q_{12}(x),q_{12}(x)+q_{13}(x))\\
  &\ldots\\
  \Gamma_{1N}(x)&=[\sum_{j=1}^{N-1}q_{1j}(x), q_{1}(x))\\
  \Gamma_{21}(x)&=[q_1(x),q_1(x)+q_{21}(x))\\
  \Gamma_{23}(x)&=[q_1(x)+q_{21}(x),q_1(x)+q_{21}(x)+q_{23}(x))\\
  &\ldots
\end{align*}
and so on. Therefore, we obtain a sequence of consecutive, left-closed, right-open intervals $\Gamma_{ij}(x)$, each having length $q_{ij}(x)$. For convenience of notation, we set $\Gamma_{ii}(x)=\emptyset$ and $\Gamma_{ij}(x)=\emptyset$ if $q_{ij}(x)=0$.
Define a function $h:\R^n\times \S\times \R\ra \R$ by
\[h(x,i,z)=\sum_{l\in \S} (l-i)\mathbf 1_{\Gamma_{il}(x)}(z).\]
Then the process $(\La(t))$ can be expressed by the SDE
\begin{equation}
  \label{1.3}\d \La(t)=\int_{[0,M]} h(X(t),\La(t-),z)N_1(\d t,\d z),
\end{equation}
where $M=N(N-1) H$, $N_1(\d t,\d z)$ is a Poisson random measure with intensity $\d t\times \mathbf{m}(\d z)$, and $\mathbf{m}(\d z)$ is the Lebesgue measure on $[0,M]$.
Let $p_1(t)$ be the stationary point process corresponding to  Poisson random measure $N_1(\d t,\d z)$. Due to the finiteness of $\mathbf{m}(\d z)$ on $[0, M]$, there is only finite number of jumps of the process $p_1(t)$ in each finite time interval. Let $ 0=\varsigma_0<\vsig_1<\ldots<\vsig_n<\ldots$ be the enumeration of all jumps of $p_1(t)$. It holds that $\lim_{n\ra \infty}\vsig_n=+\infty$ almost surely. Due to \eqref{1.3}, it follows that, if $\La(0)=i$,
\begin{equation}
  \label{1.4} \La(\vsig_1)=i+\sum_{l\in\S}(l-i) \mathbf{1}_{\Gamma_{il}(X(\vsig_1))}(p_1(\vsig_1)).
\end{equation}
This yields that $(\La(t))$ has a jump at $\vsig_1$ (i.e. $\La(\vsig_1)\neq \La(\vsig_1-)$) if $p_1(\vsig_1)$ belongs to the interval $\Gamma_{il}(X(\vsig_1))$ for some $l\neq i$. At any other cases, $(\La(t))$ admits no jump. So the set of jumping times of $(\La(t))$ is a subset of $\{\vsig_1,\vsig_2,\ldots\}$. This fact will be used below without mentioning it again.

To make our computation below more precise, we give out an explicit construction of the probability space used in the sequel.
Let
\[\Omega_1=\{\omega|\ \omega:[0,\infty)\ra \R^n \ \text{is continuous with $\omega(0)=0$}\},
\] which is endowed with the locally uniform convergence topology and the Wiener measure $\p_1$ so that the coordinate process $W(t,\omega):=\omega(t)$, $t\geq 0$, is a standard $n$-dimensional Brownian motion.
Let $(\Omega_2,\mathscr F_2,\p_2)$ be a probability space and $\Pi_{\R}$ be the totality of point functions on $\R$. For a point function $(p(t))$, $D_{p}$ denotes its domain, which is a countable subset of $[0,\infty)$. Let $p_1:\Omega_2\ra \Pi_{\R}$ be a Poisson point process with counting measure $N_1(\d t,\d z)$ on $(0,\infty)\times [0,M]$ defined by
\begin{equation}
  \label{1.5} N_1((0,t)\times U)=\#\{s\in D_{p_1}|\ s\leq t,\ p_1(s)\in U\},\ t>0,\ U\in \mathscr B([0,M]),
\end{equation}
and its intensity measure is $\d t\times\mathbf{m}(\d z)$. Set $(\Omega,\mathscr F,\p)=(\Omega_1\times\Omega_2,\mathscr B(\Omega_1)\times \mathscr F_2,\p_1\times\p_2)$, then under $\p=\p_1\times\p_2$, for $\omega=(\omega_1,\omega_2)$, $t\mapsto \omega_1(t)$ is a Wiener process, which is independent of the Poisson point process $t\mapsto p_1(t,\omega_2)$. Throughout this work, we will work on this probability space $(\Omega,\mathscr F,\p)$.

\subsection{Two points state space case}
To emphasize the idea, we restrict ourself to the situation that the state space $\S$ contains only two points in this subsection. We first present an estimate on the exponential functional of the state-dependent jumping process by comparing it with a state-independent Markov chain through constructing a coupling process of $(\La(t))$ using Skorokhod's representation in \eqref{1.3}. This estimate plays an important role in controlling the evolution of this regime-switching system.

\begin{lem}[Estimate of exponential functional of $(\La(t))$] \label{t-1.1}
Let $(X(t),\La(t))$ satisfy \eqref{1.1} and \eqref{1.2} with $\S=\{1,2\}$.  Let $(\lambda_i)_{i\in \S}$ be a nondecreasing sequence, i.e. $\lambda_1\leq \lambda_2$. Set $\bar q_{12}=\sup_{x\in\R^n} q_{12}(x)$, $\bar q_{21}=\inf_{x\in\R^n} q_{21}(x)$, $\bar q_1=-\bar q_{11}=\bar q_{12}$, and $\bar q_{2}=-\bar q_{22}=\bar q_{21}$. Assume
\begin{equation}\label{con-q}
\bar q_{21}>0,\ \ \bar q_{12}+\bar q_{21}\leq q_{12}(x)+q_{21}(x)\ \text{for every $x\in \R^n$}.
\end{equation}
Set \[\bar Q_{\lambda}=\begin{pmatrix}
  -\bar q_1 & \bar q_{12}\\
  \bar q_{21} &-\bar q_2
\end{pmatrix}  +\begin{pmatrix} \lambda_1 &0\\ 0&\lambda_2\end{pmatrix}, \]
and $\bar \eta=-\max_{\gamma\in \mathrm{spec} \bar Q_\lambda} \mathrm{Re}\, \gamma.$
Then there exists a constant $C>0$ such that
\begin{equation} \label{est-1.1}
  \E \e^{\int_0^t \lambda_{\La(s)}\d s} \leq C \e^{-\bar \eta t}, \ \text{for all $t>0$.}
\end{equation}
\end{lem}

\begin{proof}
  Set $\bar \Gamma_{12}=[0,\bar q_{12})$, $\bar \Gamma_{21}=[\bar q_{12},\bar q_{12}+\bar q_{21})$, $g(1,z)=\mathbf 1_{\bar \Gamma_{12}}(z)$, and $g(2,z)=-\mathbf 1_{\bar \Gamma_{21}}(z)$. Let $(\bar \La(t))$ be the solution of the following SDE
  \begin{equation}
    \label{1.7} \d \bar \La(t)=\int_{[0,M]} g(\La(t-),z)N_1(\d t,\d z), \quad \bar \La(0)=\La(0).
  \end{equation} Then $(\bar\La(t))$ is a jumping process with the transition rate matrix $(\bar q_{ij})$. Note that the process $(\bar \La(t))$ is independent of $\omega_1\in \Omega_1$, which is a crucial  point used in the deduction below.  Recall that $\{\vsig_k; k\geq 1\}$ denotes the set of all jumps of Poisson point process $(p_1(t))$, thus the processes $(\La(t))$ and $(\bar \La(t))$ have no jumps out of the set $\{\vsig_k; k\geq 1\}$ due to the representations \eqref{1.3} and \eqref{1.7}. Hence, in order to show that almost surely $\La(t)\leq \bar \La(t)$ for all $t>0$,  we only need to show almost surely $\La(\vsig_k)\leq \bar \La(\vsig_k)$ for all $k\geq 1$.

  If $\La(0)=\bar \La(0)=1$, then
  \begin{align*}
    \La(\vsig_1)&=1+\mathbf{1}_{\Gamma_{12}(X(\vsig_1))}(p_1(\vsig_1)),\\
    \bar\La(\vsig_1)&=1+\mathbf 1_{\bar \Gamma_{12}}(p_1(\vsig_1)).
  \end{align*}
  By the definition of $\Gamma_{12}(x)$ and $\bar\Gamma_{12}$, it is easy to see that when $p_1(\vsig_1)\in \Gamma_{12}(X(\vsig_1))$, it must hold that $p_1(\vsig_1)\in \bar\Gamma_{12}$. Hence, when $\La(\vsig_1)=2$, it must hold that $\bar \La(\vsig_1)=2$. So $\La(\vsig_1)\leq \bar \La(\vsig_1)$ a.s..

  If $\La(0)=\bar\La(0)=2$, then
  \begin{align*}
    \La(\vsig_1)&=2-\mathbf 1_{\Gamma_{21}(X(\vsig_1))}(p_1(\vsig_1)),\\
    \bar \La(\vsig_1)&=2-\mathbf 1_{\bar\Gamma_{21}}(p_1(\vsig_1)).
  \end{align*}
  If $\bar \La(\vsig_1)=1$, then $p_1(\vsig_1)\in\bar \Gamma_{21}$, which implies $p_1(\vsig_1)\leq\bar q_{12}+\bar q_{21}$, and $p_1(\vsig_1)\geq \bar q_{12}\geq q_{12}(X(\vsig_1))$.
  Invoking the condition that $\bar q_{12}+\bar q_{21}\leq q_{12}(x)+q_{21}(x)$ for every $x\in \R^n$, we have $p_1(\vsig_1)\in \Gamma_{21}(X(\vsig_1))$, and hence $\La(\vsig_1)=1$. So $\La(\vsig_1)\leq \bar \La(\vsig_1)$ a.s. whatever the initial value of $\La(0)=\bar \La(0)$ is 1 or 2.
  In the same manner, we can prove that $\La(\vsig_{k+1})\leq \bar\La(\vsig_{k+1})$ a.s. if $\La(\vsig_k)=\bar\La(\vsig_k)$, $k\geq 2$.

  Now, assuming $\La(\vsig_k)=1< \bar \La(\vsig_k)=2$, we go to prove that $\La(\vsig_{k+1})\leq \bar \La(\vsig_{k+1})$ almost surely. In this case,
  \begin{align*}
    \La(\vsig_{k+1})&=1+\mathbf 1_{\Gamma_{12}(X(\vsig_{k+1}))}(p_1(\vsig_{k+1})),\\
    \bar \La(\vsig_{k+1})&=2-\mathbf 1_{\bar \Gamma_{21}}(p_1(\vsig_{k+1})).
  \end{align*}
   If $\bar \La(\vsig_{k+1})=1$, then $p_1(\vsig_{k+1})\in \bar \Gamma_{21}$, and hence $\bar q_{12}+\bar q_{21} >p_1(\vsig_{k+1})\geq \bar q_{12}\geq q_{12}(X(\vsig_{k+1}))$. Together with the condition that $\bar q_{12}+\bar q_{21}\leq q_{12}(x)+q_{21}(x)$, we get $q_{12}(X(\vsig_{k+1}))\leq p_1(\vsig_{k+1})< q_{12}(X(\vsig_{k+1}))+q_{21}(X(\vsig_{k+1}))$, which implies that $ p_1(\vsig_{k+1})\in \Gamma_{21}(X(\vsig_{k+1}))$ and further $\La(\vsig_{k+1})=1=\bar \La(\vsig_{k+1})$. If $\bar \La(\vsig_{k+1})=2$, it is trivial to see that $\La(\vsig_{k+1})\leq \bar \La(\vsig_{k+1})$ a.s.. Consequently, we obtain $\La(\vsig_{k+1})\leq \bar \La(\vsig_{k+1})$ a.s.. In all, we have proved that
   \begin{equation}\label{com-1}
   \La(t)\leq \bar\La(t)\quad a.s..
   \end{equation}

   By virtue of the nondecreasing property of $(\lambda_i)_{i\in\S}$, it follows that  $\lambda_{\La(t)}\leq \lambda_{\bar \La(t)}$ almost surely, and hence
   \[\E\e^{\int_0^t \lambda_{\Lambda(s)}\d s}\leq \E \e^{\int_0^t\lambda_{\bar \La(s)}\d s},\quad t>0.\]
   According to \cite[Proposition 4.1]{Bar}, there exists a constant $C>0$ such that
   \[\E\e^{\int_0^t \lambda_{\Lambda(s)}\d s}\leq \E \e^{\int_0^t\lambda_{\bar \La(s)}\d s}\leq C\e^{-\bar \eta t},\quad t>0,\]
   and the proof is complete.
\end{proof}

\begin{rem} In Lemma \ref{t-1.1}, the definition of the process $(\bar \La(t))$ depends on the monotonicity of $(\lambda_i)_{i\in \S}$. If $\lambda_1>\lambda_2$, in order to control the functional $\int_0^t\lambda_{\La(s)}\d s$ of $(\La(t))$ via a Markov chain, we need to modify the definition of $(\bar q_{ij})$ as follows:
\[\bar q_{12}=\inf_{x\in \R^n} q_{12}(x),\quad \bar q_{21}=\sup_{x\in \R^n} q_{21}(x).\]
Then it still holds
\[\E\e^{\int_0^t\lambda_{\La(s)}\d s}\leq \E\e^{\int_0^t\lambda_{\bar \La(s)}\d s}\leq C \e^{-\bar \eta t},\] where $\bar \eta$ is corresponding to $\bar Q_\lambda$ using the new definition of $\bar q_{ij}$ as above.
\end{rem}

The existence and uniqueness of the invariant measure for $(X(t),\La(t))$ is deduced by analyzing the convergence of its distribution in the Wasserstein distance. This idea has been used in \cite{CH} and \cite{Sh15b}. The dependence of the transition rate of $(\La(t))$ on the process $(X(t))$ makes it much difficulty to ensure the coupling process to be successful. Next, we shall introduce our coupling process for $(X(t),\La(t))$ and prove it to be successful after some necessary preparations.

Let $(X^{x,i}(t), \La^{x,i}(t))$  and $(X^{y,j}(t),\La^{y,j}(t))$ denote the solutions of \eqref{1.1} and \eqref{1.2} starting from $(x,i)$ and $(y,j)$ respectively. To estimate the Wasserstein distance between $(X^{x,i}(t), \La^{x,i}(t))$  and $(X^{y,j}(t),\La^{y,j}(t))$, we introduce the coupling by reflection as follows:
Set
\begin{equation}\label{e-1}
a(x,i)=\sigma(x,i)\sigma(x,i)^\ast,\ a(x,i,y,j)=\begin{pmatrix}
  a(x,i) &c(x,i,y,j)\\ c(x,i,y,j)& a(y,j)
\end{pmatrix}, \quad x\in \R^n,\ i\in \S,
\end{equation}
where
\[c(x,i,y,j)=\sigma(x,i)\big(\mathrm{I}-2\bar u\bar u^\ast\big)\sigma(y,j)^\ast,\]
and $\bar u=(x-y)/|x-y|$. Here $A^\ast$ denotes the transpose of the matrix $A$.
Consider the following SDEs:
\begin{equation}\label{coup-1}
\d \begin{pmatrix} X(t)\\  Y(t) \end{pmatrix} =\begin{pmatrix} b(X(t),\La(t))\\ b(Y(t),\La'(t)) \end{pmatrix} +G(X(t),\La(t),Y(t),\La'(t))\d \tilde W(t),
\end{equation}
where the matrix $G(x,i,y,j)$ satisfies $G(x,i,y,j)G^\ast(x,i,y,j)=a(x,i,y,j)$, and $(\tilde W(t))$ denotes the $2n$-dimensional Wiener process;
\begin{equation}\label{coup-2}
\begin{split}
  \d \La(t)&= \int_{[0,M]}h(X(t),\La(t-),z)N_1(\d t,\d z)\\
  \d \La'(t)&=\int_{[0,M]}h(Y(t),\La'(t-),z) N_2(\d t,\d z),
\end{split}
\end{equation} satisfying $(X(0),\La(0))=(x,i)$ and $(Y(0),\La'(0))=(y,j)$, where $N_1(\d t,\d z)$ and $N_2(\d t,\d z)$ are mutually independent Poisson random measures with intensity measure $\d t\mathbf{m}(\d z)$. The existence of solution of SDEs \eqref{coup-1} and \eqref{coup-2} can be established in the same way as \eqref{1.1} and \eqref{1.2}.  Then $(X(t),\La(t),Y(t),\La'(t))$ is known as a coupling by reflection of the processes $(X^{x,i}(t),\La^{x,i}(t))$ and $(X^{y,j}(t),\La^{y,j}(t))$.

\begin{lem}\label{t-1.2} Assume that (Q1)-(Q3) and \textbf{(A1)}, \textbf{(A2)} hold. Define
\[q^{\alpha}_{12}=\sup_{x\in \R^n} q_{12}(x),\quad q^{\alpha}_{21}=\inf_{x\in \R^n}q_{21}(x),\ \ \text{if}\  \alpha_1\leq \alpha_2;\] otherwise,
\[q^{\alpha}_{12}=\inf_{x\in\R^n} q_{12}(x),\quad q^{\alpha}_{21}=\sup_{x\in \R^n} q_{21}(x).\]
Put $q^{\alpha}_1=-q^\alpha_{11}=q^\alpha_{12}$, $q^{\alpha}_2=-q^\alpha_{22}=q^{\alpha}_{21}$.
  Set $Q^\alpha=(q_{ij}^\alpha)$, $Q_2=Q^\alpha+\mathrm{diag}(\alpha_1,\alpha_2)$.   Suppose $Q^\alpha$ is irreducible and
  \begin{equation}\label{con-1}
  \eta_\alpha:=-\max_{\zeta\in \mathrm{spec}(Q_2)} \mathrm{Re}\,\zeta>0.
  \end{equation}
Then there exists a constant $C>0$ such that
\begin{equation}
  \label{2.1}\E |X(t)-Y(t)|^2\leq C |x-y|^2 \e^{  -\eta_\alpha  t}, \quad t>0.
\end{equation}
\end{lem}

\begin{proof}
  For simplicity of notation, set
  $Z(t)=X(t)-Y(t)$.
  According to the construction of $a(x,i,y,j)$, it holds
  \begin{align*}
    &\mathrm{tr}(a(x,i,y,j))\\
    &=\mathrm{tr}\big(\sigma(x,i)\sigma(x,i)^\ast
    +\sigma(y,j)\sigma(y,j)^\ast-2\sigma(x,i)\sigma(y,j)^\ast\big)+4\frac{(x-y)^\ast}{|x-y|}
    \sigma(y,j)^\ast\sigma(x,i)\frac{(x-y)}{|x-y|}\\
    &=\|\sigma(x,i)-\sigma(y,j)\|_{\mathrm{HS}}^2+4\frac{(x-y)^\ast}{|x-y|}\sigma(y,j)^\ast
    \sigma(x,i)\frac{(x-y)}{|x-y|}.
  \end{align*}
  By \textbf{(A1)}, \textbf{(A2)} and It\^o's formula, we obtain, for any $\gamma>0$,
  \begin{equation}\label{2.2}
    \begin{split}
    \d |Z(t)|^2
    &=\big\{2\la Z(t),b(X(t),\La(t))-b(Y(t),\La'(t))\raa \\ &\quad+ \mathrm{tr}(a(X(t),\La(t),Y(t),\La'(t)))\big\}\d t+\d M_t\\
    &\leq \Big\{\alpha_{\La(t)}|Z(t)|^2 +2\la Z(t), b(Y(t),\La(t))-b(Y(t),\La'(t))\raa \\
    &\quad + 2\|\sigma(Y(t),\La(t))-
    \sigma(Y(t),\La'(t))\|_{\mathrm{HS}}^2\\ &\quad+ 4\frac{(X(t)-Y(t))^\ast}{|X(t)-Y(t)|}\sigma(Y(t),\La(t))^\ast\sigma(X(t),\La(t))
    \frac{(X(t)-Y(t))}{|X(t)-Y(t)|}\Big\}\d t  +\d M_t\\
    &\leq \big\{(\gamma+\alpha_{\La(t)})|Z(t)|^2+\frac{4C_1^2}{\gamma}+12C_1^2\big\}\d t+\d M_t,
    \end{split}
  \end{equation}
  where $(M_t)$ is a martingale with $M_0=0$. By replacing $\bar q_{ij}$ with $q_{ij}^\alpha$, similar to \eqref{1.7}, we can define a Markov chain $(\La^\alpha(t))$ with the transition rate matrix $Q^\alpha$ and satisfying $\alpha_{\La(t)}\leq \alpha_{\La^\alpha(t)}$ for all $t>0$ almost surely.
  Hence, for every $\lambda >0$,
  \begin{align*}
    &\E_{\p_1} \Big[\e^{-\lambda t} |Z(t)|^2\Big]\\
    & \leq |x-y|^2+\int_0^t (4\gamma^{-1}+12)C_1^2\e^{-\lambda s}\d s+\E_{\p_1}\int_0^t (-\lambda+\gamma+\alpha_{\La(s)})\e^{-\lambda s}|Z(s)|^2\d s\\
    &\leq |x-y|^2+\frac{ (4\gamma^{-1}+12)C_1^2}{\lambda}+\int_0^t (-\lambda+\gamma+\alpha_{ \La^{\alpha}(s)})\e^{-\lambda s}\E_{\p_1}|Z(s)|^2\d s.
  \end{align*}
  By Gronwall's inequality, we get
  \[\e^{-\lambda t} \E_{\p_1} |Z(t)|^2\leq \Big(|x-y|^2+\frac{(4\gamma^{-1}+12)C_1^2}{\lambda}\Big)
  \e^{\int_0^t(-\lambda+\gamma+\alpha_{\La^\alpha(s)})\d s}.\]
  Then, taking expectation w.r.t.\! $\p_2$ in both sides of the previous inequality and applying \cite[Proposition 4.1]{Bar}, we obtain that there exists a $C>0$ such that
  \begin{equation}\label{2.3}
  \E|Z(t)|^2 \leq C\big(|x-y|^2+\frac{(4\gamma^{-1}+12)C_1^2}{\lambda}\big)\e^{(\gamma -\eta_2) t},\quad t>0.
  \end{equation}
  By the arbitrariness of $\gamma $ and $\lambda$, letting first $\lambda\ra+\infty$ then $\gamma\downarrow 0$ in \eqref{2.3}, we obtain that
  \begin{equation}\label{2.4}
  \E |Z(t)|^2\leq C|x-y|^2\e^{  -\eta_\alpha t},
  \end{equation} and further that
  \[\sup_{t>0}\E |Z(t)|^2 <\infty\]
  due to the positiveness of $\eta_\alpha$.
\end{proof}

\begin{lem}\label{t-1.3}
Under the same assumptions and notation of Lemma \ref{t-1.2},  it holds that
\begin{equation}\label{3.1}
\sup_{t\geq 0} \E|X^{x,i}(t)|^2\leq C(1+|x|^2),\quad x\in \R^n,\, i\in \S,
\end{equation} where $C$ is a constant.
\end{lem}

\begin{proof}
Note that condition \textbf{(A1)} implies that for any $\veps>0$, there exists a constant $C_\veps>0$ such that
\begin{equation}\label{3.1.5}
2\la x,b(x,i)\raa+\|\sigma(x,i)\|_{\mathrm{HS}}^2\leq C_{\veps}+(\veps+\alpha_i)|x|^2,\quad x\in \R^n,\,i\in \S.
\end{equation}

By \eqref{3.1.5} and applying It\^o's formula to $X(t)=X^{x,i}(t)$ yields that
    \[\d |X(t)|^2\leq (C_\veps+(\veps+\alpha_{\La(t)})|X(t)|^2)\d t+2\la X(t),\sigma(X(t),\La(t))\d W(t)\raa.\]
  For every $\lambda>0$, we have
  \[\d \big[ \e^{-\lambda t} |X(t)|^2\big]\leq \!\e^{-\lambda t}\big\{\!-\!\lambda |X(t)|^2\!+\!C_\veps\!+\!(\veps\!+\!\alpha_{\La(t)})|X(t)|^2\big\}\d t+2\e^{-\lambda t}\la X(t),\sigma(X(t),\La(t))\d W(t)\raa.\]
  Taking expectation in both sides w.r.t. $\p_1$ and noting $\alpha_{\La(t)}\leq \alpha_{\La_\alpha(t)}$ a.s. by Lemma \ref{t-1.1}, we can deduce that
  \begin{equation}\label{3.3}
  \begin{split}
  \e^{-\lambda t}\E_{\p_1} |X(t)|^2&\leq |x|^2+\frac{C_\veps}{\lambda}+\int_0^t(\veps+\alpha_{\La(s)}-\lambda)\e^{-\lambda s}\E_{\p_1}|X(s)|^2\d s\\
  &\leq |x|^2+\frac{C_\veps}{\lambda}+\int_0^t(\veps+\alpha_{ \La^\alpha(s)}-\lambda)\e^{-\lambda s}\E_{\p_1} |X(s)|^2\d s.
  \end{split}
  \end{equation}
  Using Gronwall's inequality, this yields
  \[\e^{-\lambda t} \E_{\p_1}|X(t)|^2\leq (|x|^2+\frac{C_\veps}{\lambda})\e^{\int_0^t(\veps+\alpha_{ \La^\alpha(s)}-\lambda)\d s}.\]
  Therefore, by \cite[Proposition 4.1]{Bar}, there exists a constant $C$ such that
  \begin{equation}\label{3.3}
  \E|X(t)|^2\leq \big(|x|^2+\frac{C_\veps}{\lambda}\big)\E\e^{\int_0^t\veps+\alpha_{ \La^\alpha(s)}\d s}\leq C\big(|x|^2+\frac{C_\veps}{\lambda}\big)\e^{-(\eta_\alpha-\veps) t}.
  \end{equation}
  Setting $\veps=\frac 12\eta_\alpha>0$, we can deduce from \eqref{3.3} that
  \[\sup_{t\geq 0} \E|X(t)|^2 \leq C(1+|x|^2).\]
\end{proof}


\begin{lem}\label{t-1.4}
Assume that (Q1)-(Q3), \textbf{(A1)}-\textbf{(A4)} and \eqref{con-q} hold. Then the coupling  $\big(X(t),\La(t)$, $Y(t),\La'(t)\big)$ determined by \eqref{coup-1} and \eqref{coup-2} is a successful coupling, that is,
\[T:=\inf\{t>0;\ (X(t),\La(t))=(Y(t),\La'(t))\}<\infty  \ \ a.s.\]
\end{lem}

\begin{proof}
  Without loss of generality, we can assume that the condition \textbf{(A4)} holds for  $i_0=1$. Otherwise, we can rearrange the order of $\S$.

  If $(\La(0),\La'(0))\neq (1,1)$, the proof is divided into three steps. Otherwise, we can start directly from the second step below.

  \noindent\textbf{Step 1}:
  Set
  \begin{equation}\label{4.1}
  \tau=\inf\{t\geq 0; \La(t)=\La'(t)=1\},
  \end{equation}  and we shall first show the stopping time $\tau$ is almost surely finite.
 Set
 \[\bar q_{12}=\sup_{x\in\R^n}q_{12}(x),\quad \bar q_{21}=\inf_{x\in\R^n} q_{21}(x).\]
 Assume that $\bar q_{12},\bar q_{21} >0$.  Define $\bar \Gamma_{12},\,\bar \Gamma_{21}$, $g(1,z)$ and $g(2,z)$ in the same way as Lemma \ref{t-1.1}. Set
 \begin{equation}\label{4.2}
 \begin{split}
   \d \La^{(1)}(t)&=\int_{[0,M]} g(\La^{(1)}(t-),z)N_1(\d t, \d z),\quad \La^{(1)}(0)=\La(0),\\
   \d \La^{(2)}(t)&=\int_{[0,M]} g(\La^{(2)}(t-),z) N_2(\d t, \d z),\quad \La^{(2)}(0)=\La'(0).
   \end{split}
 \end{equation}
 According to \eqref{com-1} in Lemma \ref{t-1.1},  it holds almost surely $\La(t)\leq \La^{(1)}(t)$ and $\La'(t)\leq \La^{(2)}(t)$, $t\geq 0$. The mutual independence of $N_1(\d t,\d z)$ and $N_2(\d t,\d z)$ yields that $(\La^{(1)}(t))$ and $(\La^{(2)}(t))$ are also mutually independent. Put
 \[\tau'=\inf\{t\geq 0; \La^{(1)}(t)=\La^{(2)}(t)=1\}\]
 Then it is easy to see that
 \begin{equation}\label{4.3}
 \tau\leq \tau',\ \ a.s.
 \end{equation}
 $(\La^{(1)}(t), \La^{(2)}(t))$ is an independent coupling corresponding to the operator $\bar Q=(\bar q_{ij})$ and itself (cf. for instance, \cite{Chen}). Due to the irreducibility of $\bar Q$ and the finiteness of $\S\times \S$, there exists a positive constant $\theta$ such that
 \begin{equation*}
 \p(\tau'\geq t)\leq \e^{-\theta t},\quad t>0.
 \end{equation*}
 Invoking \eqref{4.3}, it holds that
 \begin{equation}\label{4.4}
 \p(\tau\geq t)\leq \p(\tau'\geq t)\leq \e^{-\theta t},\quad t>0,
 \end{equation}
  and hence $\p(\tau=\infty)=0$.

\noindent \textbf{Step 2}: Using the notation introduced in \eqref{coup-1}, let $(X^{(1)}(t),Y^{(1)}(t))$ be the solution of the following SDE:
  \begin{equation}\label{4.5}
  \d \begin{pmatrix} X^{(1)}(t)\\  Y^{(1)}(t) \end{pmatrix} =\begin{pmatrix} b(X^{(1)}(t),1)\\ b(Y^{(1)}(t),1) \end{pmatrix}\d t +G(X^{(1)}(t),1,Y^{(1)}(t),1)\d \tilde W(t),
  \end{equation} satisfying $(X^{(1)}(0), Y^{(1)}(0))=(x,y)$,
  which is the corresponding diffusion process of $(X(t), Y(t))$ in the fixed environment $(i,j)=(1,1)$. We shall use the criteria established in \cite{CL}  to verify this is a successful coupling. To estimate the coupling time, as done in \cite{CL}, we introduce the following notation:
  \begin{align*}
    A(x,y)&=a(x,1)+a(y,1)-2c(x,1,y,1),\\
    B(x,y)&=\la x-y, (b(x,1)-b(y,1))(x-y)\raa,\\
    \bar A(x,y)&=\la (x-y), A(x,y)(x-y)\raa /|x-y|^2,\ \ x\neq y.
  \end{align*}
  By the condition \textbf{(A3)}, it holds
  \begin{align*}
  \inf_{|x-y|=r}\bar A(x,y)&=\inf_{|x-y|=r}|(\sigma(x,1)-\sigma(y,1))\bar u|^2 +4 (\bar u^\ast\sigma^\ast(x,1)\bar u)(\bar u^\ast \sigma(y,1)^\ast \bar u)\\
  &\geq 4C_2^2,
  \end{align*} where $\bar u=(x-y)/|x-y|$.
  According to the condition \eqref{con-b},
  \begin{align*}
    &\sup_{|x-y|=r} \frac{\mathrm{tr}(A(x,y))-\bar A(x,y)+2B(x,y)}{\bar A(x,y)}\\
    &\leq  \sup_{|x-y|=r}\frac{\beta|x-y|^2-C_3|x-y|^p}{\bar A(x,y)}-1\leq \frac{\beta r^2-C_3 r^p}{4C_2^2}.
  \end{align*}
  Set $\alpha(r)=4C_2^2$, $\gamma(r)=\frac{\beta r^2-C_3 r^p}{4C_2^2}$, and
  \begin{gather*}
    C(r)=\exp\Big[\int_1^r\frac{\gamma(u)}{u}\d u\Big].
  \end{gather*}
  Analogous to \cite[Theorems 4.2 and 5.1]{CL}, for positive integers $\ell$ and $k$, set
  \begin{align*}
    T^{(1)}&=\inf\{t\geq 0; X^{(1)}(t)=Y^{(1)}(t)\},\\
    S_\ell&=\inf\{t\geq 0; |X^{(1)}(t)-Y^{(1)}(t)|>\ell\},\\
    T_k&=\inf\{t\geq 0; |X^{(1)}(t)-Y^{(1)}(t)|<\frac 1n\}.
  \end{align*}
  Put $T_{k,\ell}=T_k\wedge S_{\ell}$, and
  \begin{align*}
    F_{k,\ell}(r)=-\int_{1/k}^r C(s)^{-1}\Big(\int_s^\ell \frac{C(u)}{\alpha(u)}\d u\Big)\d s.
  \end{align*}
  Then it holds
  \begin{gather*}
    -\infty<F_{k,\ell}(r)\leq 0,\quad F_{k,\ell}'(r)\leq 0,\\
    F''_{k,\ell}(r)+\frac{F_{k,\ell}'(r)\gamma(r)}{r}=\frac1{\alpha(r)}.
  \end{gather*}
  Applying Dynkin's formula, we get that
  \begin{align*}
    &\E_{x,y} F_{k,\ell}(|X^{(1)}(t\wedge T_{k,\ell})-Y^{(1)}(t\wedge T_{k,\ell})|)-F_{k,\ell}(|x-y|)\\
    &=\frac 12\E_{x,y} \!\int_0^{t\wedge T_{k,\ell}}\!\!\!\bar A(X^{(1)}(s), Y^{(1)}(s)) F_{k,\ell}''(|Z^{(1)}(s)|)+ F_{k,\ell}'(|Z^{(1)}(s)|)\big[\mathrm{tr} A(X^{(1)}(s), Y^{(1)}(s))\\ &\qquad \qquad-\bar A(X^{(1)}(s),Y^{(1)}(s))+2B(X^{(1)}(s),Y^{(1)}(s)) \big]\big/|Z^{(1)}(s)|\ \d s\\
    &\geq \frac 12 \E_{x,y} \big(t\wedge T_{k,\ell}\big)
  \end{align*}
  Letting $t\ra \infty$, this yields that
  \[\E_{x,y} \,T_{k,\ell} \leq -2 F_{k,\ell}(|x-y|).\]
  Set
  \[F(r)=\lim_{k\ra \infty}\lim_{\ell\ra \infty} F_{k,\ell}=-\int_0^r C(s)^{-1}\Big(\int_s^\infty\frac{C(u)}{\alpha(u)}\d u\Big)\d s.\]
  Letting $\ell\ra \infty$ and then $k\ra \infty$, we obtain
  \begin{equation}\label{4.6}
  \E_{x,y} \,T^{(1)}\leq -2 F(|x-y|)
  \end{equation}
  It is simple to check that
  \[C(s)^{-1}\int_s^\infty \frac{C(u)}{\alpha(u)}\d u\sim s^{1-p},\quad \text{as $s\ra \infty$}.\]
  As $p>2$, this yields that
  \[\lim_{r\ra \infty} F(r)=-\int_0^\infty C(s)^{-1}\Big(\int_s^{\infty}\frac{C(u)}{\alpha(u)}\d u\Big) \d s>-\infty,\]
  and further
  \begin{equation}\label{4.7}
  \sup_{x,y\in \R^n}\E_{x,y} T^{(1)}<\infty.
  \end{equation}
  Therefore, by Chebyshev's inequality,  there exists $t_0>0$ such that for any initial point $(x,y)$,
  \begin{equation}\label{4.8}
  \p(T^{(1)} <t_0)\geq \frac 12.
  \end{equation}

  \noindent \textbf{Step 3}: Define
  \[\eta_1=\inf\{t\geq 0; \ (\La(t),\La'(t))\neq (\La(0),\La'(0))\}.\]
  By \eqref{coup-2} and the property of Poisson point process, it is easy to see that
  $\eta_1\geq \vsig^{(1)}_1\wedge \vsig^{(2)}_1$, where $\vsig^{(1)}_1$ and $\vsig^{(2)}_1$  are the first jumping times of the Poisson point processes $(p_1(t))$ and $(p_2(t))$ respectively. So
  \[\p(\eta_1\geq t)\geq \p(\vsig^{(1)}_1 \geq t)\p(\vsig^{(2)}_1\geq t)=\e^{-2Mt},\ \ t>0.\]
  Set $\zeta_0=0$,
  \begin{align*}
    \zeta_1&=\inf\{t\geq 0; (\La(t),\La'(t))\neq (\La(0),\La'(0))\},\\
    \zeta_{2m}&=\inf\{t\geq \zeta_{2m-1}; (\La(t),\La'(t))=(\La(0),\La'(0))\},\\
    \zeta_{2m+1}&=\inf\{t\geq \zeta_{2m}; (\La(t),\La'(t))\neq (\La(0),\La'(0))\}, \ \ m=1,2,\ldots.
  \end{align*}
  We have the following estimate on the coupling time $T$:
  \begin{equation}\label{4.9}
  \begin{split}
    \p^{(x,1,y,1)}(T\in [0,\zeta_1))&= \p^{(x,1,y,1)}(T\in [0,\eta_1))\\
    & \geq \p^{(x,1,y,1)}(\eta_1\geq t_0)\p^{(x,1,y,1)}( T\in [0,\eta_1)|\eta_1\geq t_0)\\
    &\geq \p^{(x,1,y,1)}(\eta_1\geq t_0)\p^{(x,y)}(T^{(1)}<t_0)\\
    &\geq \e^{-2Mt_0}/2=:\delta_2>0,
  \end{split}
  \end{equation}
  where $t_0$ is determined by \eqref{4.8} and is independent of the initial point of  $(X^{(1)}(t), Y^{(1)}(t))$.
   Therefore,
   \begin{equation}\label{4.10}
   \begin{split}
   \p^{(x,i,y,j)}(T=\infty)&=\p^{(x,i,y,j)}
   \big(\mathbf 1_{\{\tau<\infty\}}\p^{(X(\tau),\La(\tau),Y(\tau),\La'(\tau))}(T=\infty)\big)\\
   &\leq \p^{(x,i,y,j)}\Big(\mathbf 1_{\{\tau<\infty\}}\p^{(X(\tau),1,Y(\tau),1)}
   \big(T\not\in\bigcup_{m=0}^K[\zeta_{2m},\zeta_{2m+1})\big)\Big)\\
   &\leq \p^{(x,i,y,j)}\Big(\mathbf 1_{\{\tau<\infty\}}\p^{(X(\tau),1,Y(\tau),1)}\big(
   T\not\in \bigcup_{m=0}^{K-1}[\zeta_{2m},\zeta_{2m+1})\big)\\
   &\hspace{2cm}\cdot \p^{(X(\zeta_{2K}),1,Y(\zeta_{2K}),1)}
   \big(T\not\in [0,\zeta_1)\big)\Big)\\
   &\leq \p^{(x,i,y,j)}\Big(\mathbf 1_{\{\tau<\infty\}}\p^{(X(\tau),1,Y(\tau),1)}\big(
   T\not\in \bigcup_{m=0}^{K-1}[\zeta_{2m},\zeta_{2m+1})\big)\Big)(1-\delta_2)\\
   &\leq (1-\delta_2)^{K+1},
   \end{split}
   \end{equation}
   where in the last step we have used the estimate \eqref{4.9} recursively. Letting $K$ tend to $\infty$, we finally get the desired estimate that $\p^{(x,i,y,j)}(T=\infty)=0$, and complete the proof.
\end{proof}

\begin{rem}
  In \cite{XS}, together with F. Xi, we have discussed the question on the existence of successful couplings for state-dependent regime-switching processes. In that work, we imposed a condition (Assumption 2.4 (i) therein) which means that for every fixed environment the corresponding coupling process is successful uniformly relative to the initial points in some sense.  Here, Lemma \ref{t-1.4} weakens this condition to assume only that there exists at least a fixed environment so that the corresponding coupling process in this fixed environment is successful uniformly with respect to initial points.
\end{rem}

Now we introduce the Wasserstein distance used in this work.
Set
\[\rho((x,i),(y,j))=\mathbf 1_{i\neq j}+|x-y|,\ \ x,\,y\in \R^n,\ i,\,j\in\S.\]
The Wasserstein distance between every two probability measures $\nu_1,\,\nu_2$ on $\R^n\times\S$ is defined by
\begin{equation}\label{wass}
W_\rho(\nu_1,\nu_2)=\inf_{\pi\in \mathscr{C}(\nu_1,\nu_2)}\Big\{\int_{(\R^n\times\S)^2}\!\!\rho((x,i),(y,j))\d \pi((x,i),(y,j))\Big\},
\end{equation}
where $\mathscr C(\nu_1,\nu_2)$ denotes the set of all couplings of $\nu_1$ and $\nu_2$ on $(\R^n\times\S)^2$. This kind of Wasserstein distance has been used in \cite{Sh15b} to investigate the recurrent property of regime-switching diffusion process. \cite{CH} used further a truncation from above on $\rho$ to define the Wasserstein distance.

\begin{thm}\label{t-inv} Let $(X(t), \La(t))$ be the solution of \eqref{1.1} and \eqref{1.2} with initial value $(X(0),\La(0))=(x,i)$. Denote the distribution of $(X(t),\La(t))$ with initial value $(X(0),\La(0))=(x,i)$ in $\R^n\times\S$ by $\delta_{(x,i)}P_{t}$ for $t\geq 0$. Assume (Q1)-(Q3), and \textbf{(A1)}-\textbf{(A4)} hold. Suppose $Q^\alpha$ defined as in Lemma \ref{t-1.1} is irreducible and \eqref{con-1} hold.
Then  there exists a unique invariant probability measure $\mu$ on $\R^n\times \S$ such that $\mu P_{t}=\mu$ for every $t>0$, every $(x,i)\in \R^n\times\S$, and
\[\lim_{t\ra \infty}W_\rho(\delta_{(x,i)}P_{t},\mu)=0\quad \text{for any $(x,i)\in\R^n\times \S$}.\]
\end{thm}

\begin{proof}
  In order to estimate the Wasserstein distance between $\delta_{(x,i)}P_t$ and $ \delta_{(y,j)} P_t$ with $i\neq j$, we use the coupling process determined by \eqref{coup-1} and \eqref{coup-2}.

 For $\kappa\in (0,1)$, it holds that
 \begin{align*}
   &W_{\rho}(\delta_{(x,i)}P_t, \delta_{(y,j)}P_t)\leq \E\big[|X(t)-Y(t)|+\mathbf 1_{\La(t)\neq \La'(t)}\big]\\
   &=\E\big[\big( |X(t)-Y(t)|+\mathbf 1_{\La(t)\neq \La'(t)} \big)\mathbf 1_{\{\tau <\kappa t\}}\big]\\
    &\quad + \E\big[ \big( |X(t)-Y(t)| +\mathbf 1_{\La(t)\neq \La'(t)}\big)\mathbf 1_{\{\tau \geq \kappa t\}}\big]\\
   &\leq \E\big[\mathbf 1_{\{\tau<\kappa t\}}\E\big[\big(|X(t)-Y(t)| +\mathbf 1_{\La(t)\neq \La'(t)}\big)\big|\mathscr F_{\tau}\big]\big]\\
   &\quad +\E\big[( 1+|X(t)|+|Y(t)|)\mathbf 1_{\{\tau\geq \kappa t\}}\big]\\
   &\leq  \E\big[\mathbf 1_{\{\tau\leq \kappa t\}} \E[|X(t)-Y(t)|\big|\mathscr F_{\tau}]\big]+\E\big[\mathbf 1_{\{\tau\leq \kappa t\}}\mathbf 1_{\La(t)\neq \La'(t)}\big]\\
   &\quad +\sqrt{\E[( 1+|X(t)|+|Y(t)|)^2]}\sqrt{\p(\tau\geq \kappa t)}\\
   &\leq  c(1+|x|+|y|)(\e^{-\frac 12 \theta \kappa t}+\e^{-\frac{\eta_\alpha}{2}(1-\kappa)t})+\E\big[\mathbf 1_{\{\tau\leq \kappa t\}}\mathbf 1_{\La(t)\neq \La'(t)}\big],
 \end{align*}
 where in the last  step we have used the estimates \eqref{4.4}, Lemma \ref{t-1.2}, and Lemma \ref{t-1.3}.

 Note that after the stopping $\tau$, the processes $(\La(t))$ and $(\La'(t))$ do not necessarily move together due to the dependence of  $(q_{ij}(x))$ on the component $x$. But, after the coupling time $T$ given in Lemma \ref{t-1.4}, the processes $(X(t),\La(t))$ and $(Y(t),\La'(t))$ will move together.  This difficulty does not exist for state-independent switching, and we can get exponential convergence of the Wasserstein distance between $\delta_{(x,i)} P_t$ and $\delta_{(y,j)} P_t$. Refer to \cite{Sh15b} and \cite{BSY} for more details.

 However, under the help of Lemma \ref{t-1.4}, we have
 \begin{equation}\label{5.1}
 \begin{split}
   \E\big[\mathbf 1_{\{\tau\leq \kappa t\}}\mathbf 1_{\La(t)\neq \La'(t)}\big]&\leq \E\big[\mathbf 1_{\{T\geq t\}} \mathbf 1_{\La(t)\neq \La'(t)}\big]
   +\E\big[\mathbf 1_{\{T<t\}}\mathbf 1_{\La(t)\neq \La'(t)}\big]\\
   &\leq \E[\mathbf 1_{\{T\geq t\}}]\longrightarrow 0,\quad \text{as $t\ra \infty$}.
 \end{split}
 \end{equation}
 Consequently, we have
 \begin{equation}\label{5.2}
 \lim_{t\ra \infty}W_{\rho}(\delta_{(x,i)}P_t, \delta_{(y,j)}P_t)=0.
 \end{equation}
According to Lemma \ref{t-1.3}, $\E[|X_t|^2]$ is bounded for all $t>0$, which yields that the family of probability measures $(\delta_{(x,i)}P_t)_{t>0}$ is tight. Moreover, this yields that $(\delta_{(x,i)}P_t)_{t>0}$ is uniformly integrable w.r.t. the Euclidean metric $|\cdot|$ in $\R^n$. Hence, $(\delta_{(x,i)}P_t)_{t>0}$ is compact in $\mathscr P(\R^n)$ w.r.t. the Wasserstein distance $W_\rho$ (cf. for instance, \cite[Proposition 7.1.5]{AGS}). There exists a subsequence $(\delta_{(x,i)}P_{t_k})_{k\geq 1}$ with $t_k\ra \infty$ as $k\ra \infty$ converging to some probability measure $\mu$ on $\R^n$.
Moreover, \eqref{5.2} implies that for all $(y,j)\in \R^n\times \S$, $\delta_{(y,j)}P_t$ converges in $W_\rho$-metric to $\mu$ as $k\ra \infty$, and further that $\nu_0P_{1,t_k}\!:=\sum_{j\in\S}\int_{\R^n} P_{1,t}^{y,j} \nu_0(\d y)$ converges in $W_\rho$-metric to $\mu$ for every probability measure $\nu_0$ on $\R^n$ satisfying $\int_{\R^n}|y|\nu_0(\d y)<\infty$. Invoking Lemma \ref{t-1.2}, we get that for every $s>0$, $\delta_{(x,i)}P_sP_{t_k}$ converges in $W_\rho$-metric to $\mu$. Since $\delta_{(x,i)}P_sP_{t_k}=\delta_{(x,i)}P_{t_k}P_s$ and the latter term converges weakly to $\mu P_s$, this yields that $\mu P_s=\mu$. Hence, $\mu$ is the unique invariant measure of the process $(X(t),\La(t))$.
\end{proof}

\subsection{General finite state space}
In this part, we extend our results in last subsection to regime-switching processes in a general finite state space. However, we need to assume further that the jumping process is of a birth-death type, i.e. $q_{ij}(x)=0$ for all $i,\,j\in\S$ with $|i-j|\geq 2$ for every $x\in \R^n$.

Lemma \ref{t-1.1} is the key point to extend Theorem \ref{t-inv} to deal with regime-switching diffusions in a general state space, since it provides a control of the  state-dependent jumping process $(\La(t))$ via a state-independent Markov chain $(\bar \La(t))$. Using this technique, we also ensure that the coupling process of state-dependent jumping process $(\La(t),\La'(t))$ can always meet some fixed point in $\S\times\S$, then further guarantee the coupling to be a successful coupling. In the following, we provide the extension of Lemma \ref{t-1.1} and give out its proof. However, the corresponding extensions of Lemmas \ref{t-1.2}, \ref{t-1.3}, \ref{t-1.4} can be established in a completely similar way, and hence are omitted.

\begin{lem}\label{t-2.1}
Assume $q_{ij}(x)=0$ for every $i,\,j\in \S$ with $|i-j|\geq 2$ and every $x\in \R^n$.
Let $(\lambda_i)_{i\in\S}$ be a nondecreasing sequence. Set
$\bar q_{i,i+1}=\sup_{x\in \R^n} q_{i,i+1}(x)$, $\bar q_{i+1,i}=\inf_{x\in\R^n} q_{i+1,i}(x)$, $\bar q_i=-\bar q_{ii}=\sum_{j\neq i} \bar q_{ij}$ for $i\in\S$. Suppose that the matrix $(\bar q_{ij})$ is irreducible.
Assume
\begin{equation}\label{m1}
\begin{split}
&\text{for $1\leq i\leq N-2$,}\ q_{i,i+1}(x)+q_{i+1,i}(x)\ \text{is independent of $x$},\\ \ &\bar q_{N-1,N}+\bar q_{N,N-1}\leq q_{N-1,N}(x)+q_{N,N-1}(x),\ \forall \,x\in\R^n.
\end{split}
\end{equation}
Set
\[\bar Q_{\lambda}=(\bar q_{ij})+\mathrm{diag}(\lambda_1,\ldots,\lambda_N),\]
where $\mathrm{diag}(\lambda_1,\ldots,\lambda_N)$ denotes the diagonal matrix generated by the vector $(\lambda_1,\ldots,\lambda_N)$.
Set
\[\bar \eta=-\max_{\gamma\in \mathrm{spec}\,\bar Q_\lambda} \mathrm{Re}\,\gamma.\]
Then there exists a constant $C>0$ such that
\begin{equation}\label{m2}
\E\e^{\int_0^t\lambda(s)\d s}\leq C\e^{-\bar \eta t},\ \ \text{for all $t>0$.}
\end{equation}
\end{lem}

\begin{proof}
 Corresponding to $(\bar q_{ij})$, we can define $\bar \Gamma_{ij}$ similarly to $\Gamma_{ij}(x)$ as follows: $\bar \Gamma_{12}=[0,\bar q_{12})$, $\bar \Gamma_{21}=[\bar q_{12},\bar q_{12}+\bar q_{21})$, $\bar \Gamma_{23}=[\bar q_{12}+\bar q_{21},\bar q_{12}+\bar q_{21}+\bar q_{23})$,
 \begin{align*}
   \bar \Gamma_{i,i-1}=\big[\sum_{j=1}^{i-1} \bar q_j, \sum_{j=1}^{i-1} \bar q_j+\bar q_{i,i-1}\big),\quad \bar \Gamma_{i,i+1}=\big[\sum_{j=1}^{i-1}\bar q_j+\bar q_{i,i-1},\sum_{j=1}^i\bar q_j\big),\quad i\geq 3.
 \end{align*}
 Set
 \[\bar h(i,z)=\sum_{\ell\in \S}(\ell-i)\mathbf 1_{\bar \Gamma_{i\ell}}(z),\]
 and
 \begin{equation}\label{m3}
 \d \bar \La(t)=\int_{[0,M]}\bar h(\bar \La(t-),z)N_1(\d t,\d z),\quad \bar \La(0)=\La(0),
 \end{equation} then $(\bar \La(t))$ is a continuous time Markov chain with transition rate matrix $(\bar q_{ij})$. Recall that $(\vsig_k)$ denotes the jumping time of the Poisson point process $(p_1(t))$, so
 \begin{equation}\label{m4}
 \bar \La(\vsig_{k+1})=\bar \La(\vsig_k)+\mathbf 1_{\bar \Gamma_{\bar \La(\vsig_k),\bar \La(\vsig_k)+1}}(p_1(\vsig_{k+1}))-\mathbf 1_{\bar \Gamma_{\bar \La(\vsig_k),\bar \La(\vsig_k)-1}}(p_1(\vsig_{k+1})).
 \end{equation}

 Note that \eqref{m1} implies
 \begin{equation}\label{m5}
 \bar q_{i,i+1}+\bar q_{i+1,i}=q_{i,i+1}(x)+q_{i+1,i}(x), \quad \forall\,x\in \R^n, \ 1\leq i\leq N-2.
 \end{equation}
 Indeed, denote by $e_i=q_{i,i+1}(x)+q_{i+1,i}(x)$ for $1\leq i\leq N-2$. Then, by the definition of $\bar q_{ij}$, for any $\veps>0$, there exists $x_\veps, x'_\veps\in \R^n$ such that \begin{align*}
   \bar q_{i,i+1}+\bar q_{i+1,i}&\leq  q_{i,i+1}(x_\veps)+\inf_{x\in\R^n}q_{i+1,i}(x)+\veps\\
   &\leq q_{i,i+1}(x_\veps)+q_{i+1,i}(x_\veps)+\veps\\
   &=e_i+\veps,
 \end{align*}
 and \begin{align*}
   \bar q_{i,i+1}+\bar q_{i+1,i}&\geq \sup_{x\in\R^n} q_{i,i+1}(x)+q_{i+1,i}(x'_{\veps})-\veps\\
   &\geq q_{i,i+1}(x'_\veps)+q_{i+1,i}(x'_\veps)-\veps\\
   &=e_i-\veps.
 \end{align*} Letting $\veps\downarrow 0$, we obtain \eqref{m5}.

 Moreover, by \eqref{m5} and the definition of $\bar q_{ij}$, it holds that for every $x\in \R^n$,
 \begin{align*}
&\sum_{j=1}^{i-1} q_j(x)+q_{i,i-1}(x)=\sum_{j=1}^{i-1}\big(q_{j,j+1}(x)+q_{j+1,j}(x)\big)\\
 &=\sum_{j=1}^{i-1}\big(\bar q_{j,j+1}+\bar q_{j+1,j}\big)=\sum_{j=1}^{i-1} \bar q_j+\bar q_{i,i-1}, \qquad 2\leq i\leq N-1,
 \end{align*}
 and
 \begin{align*}
   \sum_{j=1}^i q_j(x)=\sum_{j=1}^{i-1} \big(\bar q_{j,j+1}+\bar q_{j+1,j}\big)+q_{i,i+1}(x)\leq \sum_{j=1}^{i-1} \big(\bar q_{j,j+1}+\bar q_{j+1,j}\big)+\bar q_{i,i+1}, \quad 1\leq i\leq N-1.
 \end{align*}
 Therefore, for every $x\in\R^n$,
 \[\Gamma_{i,i+1}(x)\subset \bar \Gamma_{i,i+1},\  1\leq i\leq N-1;\   \bar \Gamma_{i,i-1}\subset\Gamma_{i,i-1}(x),\   2\leq i\leq N.\]

 \noindent \textbf{Case 1}: $\bar \La(\vsig_k)=\La(\vsig_k)$. For simplicity of notation, denote $\bar \La(\vsig_k)=\La(\vsig_k)=i$.\\[-2em]
 \begin{itemize}
  \item If $\La(\vsig_{k+1})=i+1$, then it must hold $p_1(\vsig_{k+1})\in \Gamma_{i,i+1}(X(\vsig_{k+1}))$, and further $p_1(\vsig_{k+1})\in \bar \Gamma_{i,i+1}$.
    Thanks to \eqref{m4}, $\bar \La(\vsig_{k+1})=i+1=\La(\vsig_{k+1})$.
  \item If $\bar \La(\vsig_{k+1})=i-1$, then $p_1(\vsig_{k+1})\in \bar \Gamma_{i,i-1}$. As $\bar \Gamma_{i,i-1}\subset \Gamma_{i,i-1}(X(\vsig_{k+1}))$, we have $p_1(\vsig_{k+1})\in \Gamma_{i,i-1}(X(\vsig_{k+1}))$ and $\La(\vsig_{k+1})=i-1$. Therefore, $\La(\vsig_{k+1})=\bar \La(\vsig_{k+1})=i-1$.
  \end{itemize}
  Consequently, if $\bar \La(\vsig_k)=\La(\vsig_k)$, we always have $\bar\La(\vsig_{k+1})\geq \La(\vsig_{k+1})$.

  \noindent\textbf{Case 2}: $\bar \La(\vsig_k)>\La(\vsig_k)$. As the processes $(\La(t))$ and $(\bar\La(t))$ can both jump forward or backward at most 1, we only need to consider the situation that $\La(\vsig_k)=i-1$ and $\bar \La(\vsig_{k})=i$ for some $i\in\S$. For other cases, it obviously holds $\bar \La(\vsig_{k+1})\geq \La(\vsig_{k+1})$. \\[-2em]
  \begin{itemize}
    \item If $\La(\vsig_{k+1})=i+1$, then $p_1(\vsig_{k+1})\in \Gamma_{i,i+1}(X(\vsig_{k+1}))$, and hence $p_1(\vsig_{k+1})\in \bar \Gamma_{i,i+1}$. This implies that $\bar \La(\vsig_{k+1})=\bar \La(\vsig_{k})=i+1$.
  \end{itemize}
  Therefore, when $\bar \La(\vsig_k)>\La(\vsig_k)$, it must hold $\bar\La(\vsig_{k+1})\geq \La(\vsig_{k+1})$.

  According to the previous discussion, and invoking the monotonicity of $(\lambda_i)_{i\in\S}$, it holds $\dis \E\e^{\int_0^t\lambda_{\La(s)}\d s}\leq \E\e^{\int_0^t\lambda_{\bar \La(s)}\d s}$.
  Applying \cite[Proposition 4.1]{Bar}, there exists a constant $C>0$ such that
  \begin{equation*}
    \E\e^{\int_0^t\lambda_{\La(s)}\d s}\leq \E\e^{\int_0^t\lambda_{\bar \La(s)}\d s}\leq C\e^{-\bar \eta t},\  t>0.
  \end{equation*}
  The proof is complete.
\end{proof}

Based on Lemma \ref{t-2.1}, we can obtain our main result in this subsection:
\begin{thm}\label{t-2.2}
Let $(X(t),\La(t))$ be the solution of \eqref{1.1} and \eqref{1.2} with $N<\infty$. Assume (Q1)-(Q4), (A1)-(A4) hold and $\alpha_1\leq \alpha_2\leq \cdots\leq \alpha_N$. $\bar Q=(\bar q_{ij})$ is defined as in Lemma \eqref{t-2.1}. We assume $\bar Q$ is irreducible and satisfies the condition \eqref{m1}. Then there exists a unique invariant probability measure $\mu$ on $\R^n\times \S$ such that $\mu P_t=\mu$ for every $t>0$, and
\[\lim_{t\ra \infty} W_{\rho}(\delta_{(x,i)} P_t,\mu)=0\quad \text{for any $(x,i)\in \R^n\times \S$.}\]
\end{thm}

The proof of this theorem is omitted since it is similar to that of Theorem \ref{t-inv}.

\section{Euler-Maruyama's approximation}

Due to the complexity of the regime-switching systems, numerical approximation is frequently an important alternative of closed-form solutions of such systems. Being extremely important, numerical methods have drawn much attention. Starting from the work \cite{YM04}, numerical approximation of state-independent regime-switching processes has been studied. See also \cite{MYY}.  Besides, the approximation of the invariant measures was investigated in \cite{BSY}. Unlike the state-independent regime-switching diffusions, less result is known for the state-dependent case since the transition rate matrix of the switching process is different at every jumping step due to its dependence on the continuous-state process. To overcome the complex caused by the mixture of $(\La(t))$ and $(X(t))$, \cite{YMYC} used the local analysis and weak convergence to construct a sequence of discrete-time jumping process to approximate the state-dependent regime-switching diffusions. Their approximation sequence is different to the usual time-discretizing EM's approximation sequence, and using this method the order of error is hard to be estimated.  In this work, we shall investigate the time-discretizing EM's approximation of the state-dependent \rsdp, and show its strong convergence in $L^1$-norm. The order of error is estimated which is consistent with that obtained in \cite{YM04} for state-independent \rsdp\, in suitable sense. Our approach relies on the refined estimate of switching process based on Skorokhod's representation of jumping process.

Consider the following EM's approximate solution to equations \eqref{1.1} and \eqref{1.2}: for $\de\in(0,1)$,
\begin{equation}\label{s-3.1}
\d Y(t)= b(Y(t_\de),\La'(t_\de))\d t+\sigma(Y(t_\de),\La'(t_\de))\d W(t),
\end{equation}
\begin{equation}\label{s-3.2}
\La'(t)=i+\int_0^t\int_{[0,M]}h(Y(s_\de),\La'(s-),z)N_1(\d s,\d z),
\end{equation}
where $N_1(\d t,\d z)$ is a Poisson random measure used in \eqref{1.3} to determine the process $(\La(t))$ with $\La(0)=i$. Here and in the sequel, for the ease of notation, we use $(Y(t),\La'(t))$ instead of $(X^\delta(t),\La^\delta(t))$ to denote the EM's approximation of $(X(t),\La(t))$ for some given $\delta$. Then, by Skorokhod's representation, it holds
\begin{equation}\label{s-3.2b}
  \p(\La'(t+\Delta)=k|\La'(t)=j, \,Y(t_\de)=y)=\begin{cases}
  q_{jk}(y)\Delta+o(\Delta),& k\neq j,\\
  1+q_{jj}(y)\Delta+o(\Delta),& k=j,
  \end{cases}
\end{equation} provided $\Delta\downarrow 0$. Set $(Y(0),\La'(0))=(X(0),\La(0))=(x,i)$.  Note that $(\La'(t))$ is a continuous time jumping process whose transition rate depends on the process $(Y(t))$. In \eqref{s-3.1}, the evolution of $Y(t)$ depends only on the embedded chain $(\La'(k\delta))_{k\geq 1}$ of the process $(\La'(t))$, which coincides with the EM's approximate solution to state-independent regime-switching process studied in \cite[Chapter 4]{MY}.

In this section, we further assume the following conditions hold:
\begin{itemize}
  \item[$\textbf{(H1)}$] $\sigma(x,i)$ is a constant matrix independent of $x$ and $i$.
  \item[$\textbf{(H2)}$]
  There exists a constant $C_{4}>0$ such that
  \[|b(x,i)-b(y,i)|\leq C_4|x-y|,\quad x,\,y\in \R^n,\ i\in \S.\]
\end{itemize}
Moreover, it is easy to see that under the conditions (Q1)-(Q3) and \textbf{(A1)}, the existence of the solution of \eqref{s-3.1} and \eqref{s-3.2} is easily established by considering recursively these equations for $t\in [k\de, (k+1)\de)$, $k\geq 0$.

The main difficult and different part to study the EM's approximation of state-dependent regime-switching diffusions against the state-independent ones is the requirement of the estimation of the term
\begin{equation}\label{o-1}
\int_0^t\p(\La(s)\neq \La'(s))\d s.
\end{equation}
We shall use Skorokhod's representation to provide a suitable estimate of \eqref{o-1}. To make our calculation clear, we present a more concrete construction of the Poisson point process $(p_1(t))$ introduced in Section 1 (cf. for example \cite[Chapter 1]{IW}).

Let $\xi_i,\,i=1,2,\ldots,$ be random variables satisfying $\p_2(\xi_i\in \d x)=\mathbf{m}(\d x)/M$.
Let $\tau_i,\,i=1,2,\ldots,$ be nonnegative random variables such that $\p_2(\tau_i>t)=\exp(-t M)$, $t\geq 0$. Suppose that $(\xi_i)$, $(\tau_i)$ are mutually independent.
Set
\begin{gather*}
  \vsig_1=\tau_1,\ \vsig_2=\tau_1+\tau_2,\ldots,\vsig_k=\tau_1+\ldots+\tau_k,\ldots,\\
  D_{p_1}=\big\{\vsig_1,\vsig_2,\ldots,\vsig_k,\ldots\big\},
\end{gather*}
and
\[p_1(\vsig_k)=\xi_k,\quad k=1,2,\ldots.\]
Then $(p_1(t))$ is a Poisson point process as desired. Set $N(t)=\#\{k;\vsig_k\leq t\}$ standing for the number of jumps of the process $(p_1(t))$ before time $t$.

We also need a preliminary result, which was first shown in \cite[Lemma 2.1]{Sh15a}.

\begin{lem}\label{t-3.1}
Assume the conditions (Q2), (Q3) hold. Denote $A\Delta B=(A\backslash B)\cup(B\backslash A)$ for Borel measurable sets $A$ and $B$, and $|A\Delta B|$ its Lebsegue measure. Then
\begin{equation}\label{o-2}
|\Gamma_{ij}(x)\Delta\Gamma_{ij}(y)|\leq \tilde K|x-y|, \ \ \text{for every $i,\,j\in \S$,}
\end{equation} where $\tilde K=2(N-1)Nc_q+1$.
\end{lem}
\begin{proof}
  For the sake of completeness, we provide a proof using the technique raised in \cite{Sh15a}. To make the idea clear, we first consider the simple case that $\S=\{1,2\}$. By (Q3), it is easy to check that
  \begin{align*}
  |\Gamma_{12}(x)\Delta\Gamma_{12}(y)|&=|q_{12}(x)-q_{12}(y)|\leq c_q|x-y|,\\
  |\Gamma_{21}(x)\Delta\Gamma_{21}(y)|&=|q_{12}(x)-q_{12}(y)|
  +|q_{12}(x)+q_{21}(x)-q_{12}(y)-q_{21}(y)|\\
  &\leq 2|q_{12}(x)-q_{12}(y)|+3|q_{21}(x)-q_{21}(y)|\\
  &\leq 3c_q|x-y|.
  \end{align*}
  For general $\S=\{1,2,\ldots,N\}$,
  \begin{align*}
    |\Gamma_{ij}(x)\Delta\Gamma_{ij}(y)|&=\big|\sum_{k=1}^{i-1} q_k(x)+\sum_{k=1,k\neq i}^{j-1} q_{ik}(x)-\sum_{k=1}^{i-1}q_k(y)-\sum_{k=1,k\neq i}^{j-1} q_{ik}(y)\big|\\
    &\quad+\big|\sum_{k=1}^{i-1} q_k(x)+\sum_{k=1,k\neq i}^jq_{ik}(x)-\sum_{k=1}^{i-1}(y)-\!\!\sum_{k=1,k\neq i}^jq_{ik}(y)\big|\\
    &\leq 2\big|\sum_{k=1}^{i-1} q_k(x)+\!\!\sum_{k=1,k\neq i}^j q_{ik}(y)-\sum_{k=1}^{i-1} q_k(y)-\!\!\sum_{k=1,k\neq i}^jq_{ik}(y)\big|+|q_{ij}(x)-q_{ij}(y)|\\
    &\leq 2(j-1)Nc_q|x-y|+c_q|x-y|\leq \tilde K|x-y|,
  \end{align*} which is the desired result.
\end{proof}

\begin{lem}\label{t-3.2}
Assume  (Q1)-(Q3), \textbf{(A1)}, \textbf{(A2)}, \textbf{(H1)} and \textbf{(H2)} hold. Let $(X(t),\La(t))$ and $(Y(t),\La'(t))$ be determined by \eqref{1.1}, \eqref{1.2} and \eqref{s-3.1}, \eqref{s-3.2} respectively. Then, for any $t>0$,  there exists a positive constant $C$ independent of   $\delta$ such that
\begin{equation}\label{s-3.9}
\int_0^t\p(\La(s)\neq \La'(s))\d s\leq C  \de^{\frac 12}+C\int_0^t\E|X(s)-Y(s)|\d s.
\end{equation}
\end{lem}

\begin{proof}
  We divide this proof into three steps.

\noindent \textbf{Step 1}: For $t\in (0,\delta]$, noting that $\La(0)=\La'(0)=i$, we have
\begin{align*}
  \p(\La(t)\neq \La'(t))
  &=\p(\La(t)\neq \La'(t), N(t)\geq 1)\\
  &=\p(\La(t)\neq \La'(t), N(t)=1)+\p(\La(t)\neq \La'(t), N(t)\geq 2).
  \end{align*}
For the first term, it is easy to check that there is some $\tilde C>0$ so that
\begin{equation}\label{s-3.9.5}
\begin{split}\p(\La(t)\neq \La'(t), N(t)\geq 2)&\leq \p(N(\de)\geq 2)=\sum_{k=2}^\infty\frac{(M\de)^k}{k!}\e^{-M\de}\\
&=1-\e^{-M\de}-M\de \e^{-M\de}\leq \tilde C\delta^2.
\end{split}
\end{equation}
To deal with the second term, since $\La(0)=\La'(0)$, we get
\begin{align*}
&\{\omega;\ \La(t)\neq \La'(t),\, N(t)=1\}\\
&= \big\{\omega;\ \tau_1<t, \tau_1+\tau_2\geq t, p_1(\tau_1)\not\in
\cup_{j\in\S}\big(\Gamma_{ij}(X(\tau_1))\cap\Gamma_{ij}(Y(\tau_{1\delta})\big)\big\}.
\end{align*}
Hence,
\begin{align*}
  &\p(\La(t)\neq \La'(t), N(t)=1)=\int_0^t\p(\La(t)\neq \La'(t),\tau_1\in \d s,\tau_2>t)\\
  &=\int_0^t\p\Big(\xi_1\not \in \bigcup_{j\in\S}\big(\Gamma_{ij}(X(s))\cap\Gamma_{ij}(Y(s_\de))\big),\tau_1\in\d s\Big) \e^{-M(t-s)}.
\end{align*}
By virtue of  Lemma \ref{t-3.1}, the Lebesgue measure of $\Gamma_{ij}(x)\Delta \Gamma_{ij}(y)$ can be controlled by $|x-y|$. Hence,
\begin{equation}\label{s-3.10}
\begin{split}
&\p(\xi_1\in \Gamma_{ij}(X(s))\Delta\Gamma_{ij}(Y(s_\de))|\tau_1\in\d s)
\leq \frac{\tilde K}{M}\E|X(s)-Y(s_\de)|,
\end{split}
\end{equation}
where we have used the fact that both $X(s)$ and $Y(s)$ are independent of $\xi_1$ under the condition $\tau_1=s$. Indeed, as $
\tau_1=s\in (0,\de)$, we have
\begin{gather*}X(s)=x+\int_0^s b(X(r),i)\d r+\int_0^s\sigma\d W(r),\\
Y(s)= x+\int_0^s b(x,i)\d r+\int_0^s\sigma \d W(r).
\end{gather*}
Above equations show that $X(s)$ and $Y(s)$ are completely determined by $(W(r),\,0\leq r\leq s)$.
Then the independence between $(W(t))$ and $\xi_1$ yields that both $X(s)$ and $Y(s)$ are independent of $\xi_1$.
Consequently, for $t\in(0,\de]$,
\begin{equation}\label{s-3.11}
\p(\La(t)\neq \La'(t))\leq \tilde C\delta^2+\tilde K \int_0^\delta \E|X(s)-Y(s_\de)|\d s.
\end{equation}

\noindent \textbf{Step 2}:
We proceed to estimating $\p(\La(k\de)\neq \La'(k\de))$ for $k\geq 2$ recursively. Denote by
$N([s,t))$ the number of jumps of $(p_1(t))$ during the period of $[s,t)$. Note that $(p_1(t))$ is a stationary point process. Set $\tau_1^\de$ be the first jumping time of $(p_1(t))$ after time $\de$, then $\tau_1^\de$ has the same law as $\tau_1$, i.e. $\p_2(\tau_1^\de >s)=\exp(-M s)$ for $s\geq 0$.  We have
\begin{align*}
  &\p(\La(2\delta)\neq \La'(2\de)|\La(\de)=\La'(\de))\\
  &=\p(\La(2\de)\neq \La'(2\de),N([\de,2\de))\geq 2|\La(\de)=\La'(\de))\\
  &\quad+\p(\La(2\de)\neq \La'(2\de),N([\de,2\de))=1|\La(\de)=\La'(\de))\\
  &\leq \p(N([\de,2\de))\geq 2)+\p(\La(2\de)\neq \La'(2\de),N([\de,2\de))=1|\La(\de)=\La'(\de))\\
  &\leq \tilde C\delta^2+\int_{\de}^{2\de} \p\big(\xi_1\not\in \cup_{j\in\S}\big(\Gamma_{\La(\de)j}(X(s))\cap\Gamma_{\La'(\de)j}(Y(s_\de)),\tau_1^\de\in \d s\big)\big)\e^{-M(2\de-s)}\\
  &\leq \tilde C\de^2+\tilde K\int_\de^{2\de} \E|X(x)-Y(s_\de)|\d s
\end{align*}
Combining with the estimation in step 1, we obtain that
\begin{align*}
  &\p(\La(2\de)\neq \La'(2\de))\\
  &\leq \p(\La(2\de)\neq \La'(2\de)|\La(\de)=\La'(\de))+\p(\La(\de)\neq \La'(\de))\\
  &\leq \tilde K\int_\de^{2\de} \E|X(s)-Y(s_\de)|\d s+\tilde C\de^2+\p(\La(\de)\neq \La'(\de))\\
  &\leq \tilde K\int_0^{2\de} \E|X(s)-Y(s_\de)|\d s+2\tilde C \de^2.
\end{align*}
Deducing recursively, we have
\begin{equation}\label{s-3.12}
\p(\La(k\de)\neq \La'(k\de))\leq \tilde K\int_0^{k\de}\E|X(s)-Y(s_\de)|\d s+k \tilde C\de^2, \ \ k\geq 1.
\end{equation}

\noindent \textbf{Step 3}:
It is standard to deduce that  $\E|X(s)-Y(s_\de)|$ is bounded for $s\in [0,T]$ from the condition \textbf{(A1)}. For $t>0$, we denote by $t_k=k\de$ for $k\leq N(t)$ and $t_{K+1}=t$ if $N(t)=K$. Then,
\begin{equation}\label{s-3.12.5}
\begin{split}
  &\int_0^t \p(\La(s)\neq \La'(s)|\La(s_\de)=\La'(s_\de))\d s\\
  &\leq \int_0^t \Big( \tilde K\int_0^{\de} \E|X(s_\de+r)-Y(s_\de)| \d r+ \tilde C\delta^2\Big) \d s\\
  &=\tilde K\sum_{k=0}^{N(t)}\int_{t_k}^{t_{k+1}} \int_0^\de\E|X(s_\de+r)-Y(s_\de)|\d r\d s+\tilde C\delta^2 t\\
  &=\tilde K\sum_{k=0}^{N(t)}\int_{t_k}^{t_{k+1}} \int_{t_k}^{t_{k+1}} \E|X(r)-Y(r_\de)|\d r\d s+\tilde C\de^2 t\\
  &=\tilde K\delta \int_0^t\E|X(s)-Y(s_\de)|\d s+\tilde C\de^2 t.
\end{split}
\end{equation}

Therefore, by \eqref{s-3.12} and  \eqref{s-3.12.5},
\begin{align*}
  &\int_0^t\p(\La(s)\neq \La'(s))\d s\\
  &=\int_0^t\!\p(\La(s)\!\neq\! \La'(s), \La(s_\de)\!=\!\La'(s_\de))\d s\!+\!\int_0^t\!\p(\La(s)\neq \La'(s),\La(s_\de)\!\neq\! \La'(s_\de))\d s\\
  &\leq \int_0^t\!\p(\La(s)\neq \La'(s)|\La(s_\de)\!=\!\La'(s_\de))\d s
  \!+\!\int_0^t\p(\La(s_\de)\neq \La'(s_\de))\d s\\
  &\leq \int_0^t\!\p(\La(s)\neq \La'(s)|\La(s_\de)\!=\!\La'(s_\de))\d s+\sum_{k=0}^K\de \p(\La(k\de)\neq \La'(k\de))\\
  &\leq \int_0^t\!\p(\La(s)\neq \La'(s)|\La(s_\de)\!=\!\La'(s_\de))\d s+\sum_{k=0}^K\de\Big[ \tilde K\int_0^{k\de}\E|X(s)-Y(s_\de)|\d s+k\tilde C\de^2\Big]\\
  &\leq \int_0^t\!\p(\La(s)\neq \La'(s)|\La(s_\de)\!=\!\La'(s_\de))\d s+\frac{\tilde C\delta(t+1)t}{2}+(t +\de )\tilde K \int_0^t\E|X(s)-Y(s_\de)|\d s\\
  &\leq \tilde K(t+2\delta)\int_0^t\E|X(s)-Y(s_\delta)|\d s+\tilde C\delta^2 t+\frac{\tilde C \de (t+1)t}{2}.
  \end{align*}
By \eqref{s-3.6} below, it holds
\[\int_0^t\E|X(s)-Y(s_\de)|\d s\leq \int_0^t\E|X(s)-Y(s)|\d s+ 2C_1\de^{\frac 12},\]
and hence
\begin{equation}\label{s-3.13}
\begin{split}
  &\int_0^t\p(\La(s)\neq \La'(s))\d s\\
  &\leq \tilde K(t+2\delta)\int_0^t\E|X(s)-Y(s)|\d s +2C_1\tilde K (t+2\delta)\de^{\frac 12}+\tilde C\delta^2 t+\frac{\tilde C(t+1)t}{2}\delta.
\end{split}
\end{equation}
This yields immediately  the estimate \eqref{s-3.9} holds for some constant $C$ independent of $\delta$.
\end{proof}

\begin{thm}\label{t-3.3}
Assume  (Q1)-(Q3), \textbf{(A1)}, \textbf{(A2)}, \textbf{(H1)} and \textbf{(H2)} hold. Let $(X(t),\La(t))$ and $(Y(t),\La'(t))$ be determined by \eqref{1.1}, \eqref{1.2} and \eqref{s-3.1}, \eqref{s-3.2} respectively. Then it holds
\begin{equation}\label{s-3.3}
\E\big[\sup_{0\leq t\leq T}|X(t)-Y(t)|\big] \leq C \de^{\frac 12},
\end{equation}  for some constant $C>0$ depending on $T$ and independent of $\de$, which yields that
\begin{equation}\label{s-3.4}
\lim_{\de\ra 0} \E\big[\sup_{0\leq t\leq T}|X(t)-Y(t)|\big]=0.
\end{equation}
\end{thm}

\begin{proof}
  Set $Z(t)=X(t)-Y(t)$ for $t\geq 0$, then $Z(0)=X(0)-Y(0)=0$  and
  \begin{align*}
    Z(t)=\int_0^t b(X(s),\La(s))-b(Y(s_\de),\La'(s_\de))\d s,\ \ t>0.
  \end{align*}
  By \textbf{(H2)}, it holds
  \begin{equation}\label{s-3.5}
  \begin{split}
    &\E\sup_{0\leq s\leq t}|Z(s)|\\
    &\leq \E\int_0^t|b(X(s),\La(s))-b(Y(s_\de),\La'(s_\de))|\d s\\
    &\leq \E\int_0^t\Big\{|b(X(s),\La(s))-b(Y(s),\La(s))|+|b(Y(s),\La(s))-b(Y(s_\de),\La(s))|\\
    &\qquad +\! |b(Y(s_\de),\La(s))\!-\!b(Y(s_\de),\La'(s))|
    \!+\!|b(Y(s_\de),\La'(s))\!-\! b(Y(s_\de),\La'(s_\de))|\Big\}\d s\\
    &\leq\E\int_0^t\Big\{ C_4\big(|Z(s)|+|Y(s)-Y(s_\de)|\big) + 2C_1\big(\mathbf 1_{\{\La(s)\neq \La'(s)\}}+\mathbf 1_{\{\La'(s)\neq \La'(s_\de)\}}\big)\Big\}\d s
  \end{split}
  \end{equation}

  By \eqref{s-3.1} and condition \textbf{(A2)}, we get
  \begin{equation}\label{s-3.6}
  \begin{split}
    \E|Y(s)-Y(s_\de)|&\leq  \E\int_{s_\de}^s |b(Y(r_\de, \La'(r_\de))|\d r  + \Big(\E\int_{s_\de}^s \|\sigma\|_{\mathrm{HS}}^2\d r\Big)^{\frac 12}\\
    &\leq C_1\de+C_1  \de^{\frac 12}\leq 2C_1\de^{\frac 12}.
  \end{split}
  \end{equation}
  For $t>0$, set $K=[t/\delta]$, $t_k=k\de$ for $k\leq K$ and $t_{K+1}=t$.
  Then, according to \eqref{s-3.2b} and (Q2),
  \begin{equation}\label{s-3.7}
  \int_0^t\E\mathbf 1_{\{\La'(s)\neq \La'(s_\de)\}}\d s =\sum_{k=0}^K\int_{t_k}^{t_{k+1}} \p(\La'(s)\neq \La'(t_k))\d s\leq H\delta t +o(\de).
  \end{equation}
 According to the Lemma \ref{t-3.2}, there exists a constant $C>0$ depending on $t$ such that
  \begin{equation}\label{s-3.8}
  \int_0^t\p (\La(s)\neq \La'(s))\d s\leq C   \de^{\frac 12}+C\int_0^t\E|Z(s)|\d s.
  \end{equation}
 Inserting \eqref{s-3.6}, \eqref{s-3.7}, \eqref{s-3.8} into \eqref{s-3.5}, we obtain
  \begin{equation*}
    \E\big[\sup_{0\leq s\leq t}|Z(t)|\big]\leq C\de^{\frac 12}+C\int_0^t\E\big[\sup_{0\leq r\leq s}|Z(r)|\big]\d s.
  \end{equation*} By Gronwall's inequality, we obtain that
  \[ \E\big[\sup_{0\leq t\leq T}|Z(t)|\big]\leq C(T) \delta^{\frac 12},\]
  which  yields the desired conclusion.
\end{proof}

\end{document}